\documentclass[10pt, a4paper,twoside]{article}
\usepackage{amssymb}
\usepackage{amsmath}
\usepackage{fancyhdr}
\usepackage{amsthm}
\usepackage{rotate}
\usepackage{color}
\usepackage{hyperref}
\usepackage{graphicx}
\usepackage{multirow}
\newtheorem{theorem}{Theorem}[section]
\newtheorem{proposition}[theorem]{propositionosition}

\newtheorem{corollary}[theorem]{Corollary}
\newtheorem{remark}[theorem]{Remark}
\newtheorem{example}[theorem]{Example}
\author{Ratan Lal, Ramjash Gurjar and Vipul Kakkar\\ \textit{vermarattan789@gmail.com, ramjashgurjar83@gmail.com},\\ \textit{vplkakkar@gmail.com}}
\title{Semidirect Product of Loops with Groups}
\date{\today}
\begin{document}
	\maketitle
	\begin{abstract}
	 	In this paper, we have studied the loops which are the semidirect products of a loop and a group. Also, the cummutant, nuclei and the center of such loops are studied.
	\end{abstract}
\section{Introduction}

A non-empty set $L$ with a binary operation $\cdot$ is called a loop if $(i)$ for all $a, b \in L$, the equations $x\cdot a = b$ and $a\cdot x = b$ have unique solution in $L$, where $x$ is an unknown and $(ii)$ there is an identity element $e\in L$ such that $a\cdot e = a = e\cdot a$. A non-empty subset $L^{\prime}$ of $L$ is said to be a subloop of $L$ if $L^{\prime}$ is itself a loop under the induced binary operation of $L$.  
\vspace{.2cm}

\noindent Recall that in group theory, if $G$ is a group with a normal subgroup $H$ and a subgroup $K$ such that $G = HK$ and $H\cap K = \{1\}$, then $G$ is called the internal semidirect product of groups. On the other hand, if $H$ and $K$ are two groups with a group homomorphism $\phi: K \longrightarrow Aut(H)$ defined by $k \mapsto \phi_{k}$, where $Aut(H)$ is the group of automorphisms of $H$, then the set $G = H\times K$ is a group with binary operation defined as 
\begin{equation}\label{s1e1}
(h_{1}, k_{1})\cdot(h_{2}, k_{2}) = (h_{1}\phi_{k_{1}}(h_{2}), k_{1}k_{2})
\end{equation}
for all $(h_{1}, k_{1}), (h_{2}, k_{2})\in G$. This is called the external semidirect product and denoted by $G = H\rtimes_{\phi} K$.
\vspace{.2cm}

\noindent	Figula and Strambach \cite{Fig07} studied the loops that are the semidirect product of two groups. Kinyon and Jones \cite{kinyon20} generalized the notion of semidirect product of groups to the semidirect product of a loop and a group. Let $N$ be a loop and $H$ be a group. Let $\phi: H \longrightarrow Sym(N)_{1}$ be a group homomorphism defined by $h\mapsto \phi_{h}$ such that $\phi_{h}(1) = 1$ for all $h\in H$, where $Sym(N)_{1}$ is the stabilizer of $1$ in $Sym(N)$. Then $L = N \rtimes_{\phi} H$ is a loop with the binary operation defined in the Equation $(\ref{s1e1})$ and is called the external semidirect product of a loop and a group. In \cite{drapal17}, \cite{gag14} and \cite{rg14}, Moufang loops have been studied as the extension of a normal Moufang subloop by a cyclic group.
\vspace{.2cm}

\noindent In Section 2, we have discussed some basic properties required. In Section 3, we have proved that the external semidirect product of a loop and group is equivalent to the internal semidirect product of a loop and a group. Also, a description of the left nucleus, middle nucleus, right nucleus, center and commutant of the loop $L$ is given. Moreover, we have shown that the middle and right nucleus  are same. 
\vspace{.2cm}

\noindent	In Section 4, we have obtained a necessary and sufficient condition for two semidirect products of a loop and a group to be isomorphic. Section 5 is devoted to the study of split meta-cyclic loops. Further, we have proved that the nuclei of a metacyclic loop are always groups.	At last, we have classified the split metacyclic loops of order $3^{3}$.

\section{Preliminaries}
Let $L$ be a loop. Let $a,b,c\in L$. By $a\cdot bc$ we mean $a(bc)$ and by $ab\cdot c$ we mean $(ab)c$. Following are some important subsets of a loop $L$.
\begin{itemize}
	\item[$(i)$]  ${N}_\lambda(L) = \{a\in L\mid a\cdot xy = ax\cdot y\}$ is called the left nucleus of $L$,
	\item[$(ii)$] ${N}_\mu(L) = \{a\in L\mid x\cdot ay = xa\cdot y\}$ middle nucleus of $L$,
	\item[$(iii)$] ${N}_\rho(L) = \{a\in L\mid x\cdot ya = xy\cdot a\}$ right nucleus of $L$,
	\item[$(iv)$] ${N}(L) = {N}_\lambda(L) \cap {N}_\mu(L) \cap {N}_\rho(L)$ is called the nucleus of $L$,
	\item[$(v)$] ${C}(L) = \{a\in L \mid a\cdot x = x\cdot a \; \forall x\in L\}$ is called the commutant of $L$,
	\item[$(vi)$] ${Z}(L) = N(L)\cap C(L)$ is called the center of $L$.
\end{itemize}

\noindent	Let $L = N\rtimes_{\phi} H$ be a loop, where $N$ is a loop, $H$ is a group and $\phi : H \longrightarrow Sym(N)_{1}$ is a group homomorphism. Then we have the following remark.

\begin{remark}\label{rem1}
	Note that there are five different bracket arrangements on four symbols. Below, we will show that these five brackets are equal in the following sense
	\begin{align*}
	(i)~~ ((n_{1}, 1)(1, h_{1}))((n_{2}, 1)(1, h_{2})) &= (n_{1}, h_{1})(n_{2}, h_{2})\\
	&= (n_{1}\phi_{h_{1}}(n_{2}), h_{1}h_{2}).\\
	(ii)~~ (n_{1}, 1)((1,h_{1})((n_{2}, 1)(1, h_{2}))) &= (n_{1}, 1)((1,h_{1})(n_{2}, h_{2}))\\
	&= (n_{1}, 1)(1\phi_{h_{1}}(n_{2}), h_{1}h_{2})\\
	&=(n_{1}\phi_{1}(\phi_{h_{1}}(n_{2})), h_{1}h_{2})\\
	&= (n_{1}\phi_{h_{1}}(n_{2}), h_{1}h_{2}).\\
	(iii)~~(n_{1}, 1)((1,h_{1})((n_{2}, 1))(1, h_{2})) &= (n_{1}, 1)((1\phi_{h_{1}}(n_{2}), h_{1})(1,h_{2}))\\
	&= (n_{1}, 1)(\phi_{h_{1}}(n_{2})\phi_{h_{1}}(1), h_{1}h_{2})\\
	&= (n_{1}\phi_{1}(\phi_{h_{1}}(n_{2})), h_{1}h_{2})\\
	&= (n_{1}\phi_{h_{1}}(n_{2}), h_{1}h_{2}).\\
	(iv)~~(((n_{1}, 1)(1,h_{1}))(n_{2}, 1))(1, h_{2}) &= ((n_{1}, h_{1})(n_{2}, 1))(1, h_{2})\\				
	&= (n_{1}\phi_{h_{1}}(n_{2}), h_{1})(1,h_{2})\\
	&= ((n_{1}\phi_{h_{1}}(n_{2}))\phi_{h_{1}}(1), h_{1}h_{2})\\
	&= (n_{1}\phi_{h_{1}}(n_{2}), h_{1}h_{2}).\\
	(v)~~((n_{1}, 1)((1,h_{1})(n_{2}, 1)))(1, h_{2}) &= ((n_{1}, 1)(1\phi_{h_{1}}(n_{2}), h_{1}))(1, h_{2})\\				
	&= (n_{1}\phi_{1}(\phi_{h_{1}}(n_{2})), h_{1})(1,h_{2})\\
	&= ((n_{1}\phi_{h_{1}}(n_{2}))\phi_{h_{1}}(1), h_{1}h_{2})\\
	&= (n_{1}\phi_{h_{1}}(n_{2}), h_{1}h_{2}).\\
	\end{align*}
\end{remark}

\begin{remark}\label{rem3}
	Note that there are two different bracket arrangements on three symbols. Below, we will show that these two brackets are equal in the following sense
	\begin{eqnarray*}
		(1,h)((n,1)(1,h)^{-1}) &= (1,h)(n,h^{-1})\\
		&= (\phi_{h}(n), hh^{-1})\\
		&= (\phi_{h}(n)\phi_{h}(1), hh^{-1})\\
		&= (\phi_{h}(n), h)(1, h^{-1})\\
		&= (1\phi_{h}(n), h1)(1, h^{-1})\\
		&= ((1,h)(n,1))(1,h)^{-1}.
	\end{eqnarray*}
\end{remark}

\begin{remark}
	Once we will prove that the internal semidirect product is isomorphic to the external semidirect product, the Remarks \ref{rem1} and \ref{rem3} will take the following forms 
	$(n_{1}h_{1})(n_{2}h_{2}) = n_{1}((h_{1}n_{2})h_{2}) = n_{1}(h_{1}(n_{2}h_{2})) = ((n_{1}h_{1})n_{2})h_{2} = (n_{1}(h_{1}n_{2}))h_{2}$ and $h(nh^{-1}) = (hn)h^{-1}$ for all $n, n_{1}, n_{2}\in N$ and $h, h_{1}, h_{2}\in H$. 
\end{remark}
Throughout the paper, we will use the Remarks \ref{rem1} and \ref{rem3} frequently without mentioning.
\begin{remark}\label{rem2}
	Let $(n,h)\in L$ be any element. Then the left inverse of $(n,h)$ is defined as ${^{-1}(n,h)} = ((\phi_{h^{-1}}(n))^{-1}, h^{-1})$ and the right inverse of $(n,h)$ is defined by $(n,h)^{-1} =  (\phi_{h^{-1}}(n^{-1}), h^{-1})$.
\end{remark}
\section{Loops as Semidirect Product}
In this section, we will prove the equivalence between the external and internal semidirect product of a loop $N$ and a group $H$.
In addition, we will give a description of the commutant, left nucleus, right nucleus, middle nucleus and center of the loop $L$. Throughout this section, $L$ will denote the loop $N\rtimes_{\phi} H$. We will replace $\cdot$ with juxtaposition suitably.

\vspace{.2cm}

\noindent	Now, using the Remarks \ref{rem1} and \ref{rem3}, for any $(n_{1}, h_{1}), (n_{2}, h_{2}) \in L$, we have
\begin{align*}
(n_{1}, h_{1})\cdot(n_{2}, h_{2}) &= ((n_{1},h_{1})(n_{2}, 1))\cdot(1,{h_{1}}^{-1}h_{1}h_{2}) \\
&= (((n_{1},h_{1})(n_{2}, 1))\cdot(1,{h_{1}}^{-1}))\cdot(1,h_{1}h_{2})\\
&=  ((n_{1}, 1)\cdot((1,h_{1})\cdot((n_{2},1)(1,{h_{1}}^{-1}))))\cdot(1,h_{1}h_{2}) \\
&= ((n_{1},1)(\phi_{h_{1}}(n_{1}), 1))\cdot(1,h_{1}h_{2}) \\
&= (n_{1}\phi_{h_{1}}(n_{1}),1)\cdot(1,h_{1}h_{2})\\
&= (n_{3},1)(1,h_{3}), 
\end{align*}
where $(n_{3},1) = (n_{1}\phi_{h_{1}}(n_{1}),1) \in N \times \{1\}$ and $(1,h_{3}) = (1,h_{1}h_{2})\in \{1\}\times H$.  Clearly, the map $nh \mapsto (n,h)$ is a bijection of the sets $NH$ and $N\times H$. Therefore, identifying $N$ with $N\times \{1\}$ and $H$ with $\{1\}\times H$, we get $h(nh^{-1}) \in N$.  Thus, we define a left action $\ast$ of $H$ on $N$ as
\begin{equation*}
h\ast n = h(nh^{-1}).
\end{equation*}
This defines a group homomorphism $\phi : H \longrightarrow Sym(N)_{1}$ by $h\mapsto \phi_{h}$, where $\phi_{h}(n) = h\ast n$ for all $h\in H$ and $n\in N$. On the other hand, the group homomorphism $\phi: H \longrightarrow Sym(N)_{1}$ will give rise to a left action of $H$ on $N$, as established in the following theorem.	
\begin{theorem}\label{s2t1}
	Let $N$ be a loop and $H$ be a group. Let $\phi$ be a group homomorphism from $H$ to $Sym(N)_1$ and $\ast$ denote the left action of $H$ on $N$ determined by $\phi$. Let $L=\{(n,h)\mid n\in N, h\in H\}$ be the set and binary operation be defined on $L$ as $(n_1,h_1)\cdot(n_2,h_2)=(n_1\phi_{h_1}(n_2), h_1h_2)$. Then
	\begin{itemize}
		\item[$(i)$] $L$ is a loop with binary operation $\cdot$ operation and $|L|=|N||H|$.
		\item[$(ii)$] The set $N'=\{(n,1)\mid n\in N\}$ is a loop and the set $H'=\{(1,h)\mid h\in H\}$ is a group and the maps $h\mapsto(1,h)$ and $n\mapsto (n,1)$ are isomorphisms from $H$ to $H'$ and $N$ to $N'$ respectively.
	\end{itemize}
	Identifying $N$ and $H$ with their isomorphic copies in $L$, we get 
	\begin{itemize}
		\item[$(a)$] $N$ is a subloop of $L$ and $H$ is a subgroup of $L$,
		\item[$(b)$] $H\cap N=\{1\}$,
		\item[$(c)$] for all $h\in H, n\in N, \phi_h(n)=h(nh^{-1})$.
	\end{itemize}
\end{theorem}
\begin{proof}
	\begin{itemize}
		\item[$(i)$] It is clear that $L$ is a loop (see \cite[pg. 1]{rg14}). Also, $|L|=|N||H|$.
		\item[$(ii)$]  Let $N'=\{(n,1)\mid n\in N\}$ and $H'=\{(1,h)\mid h\in H\}$ be two subsets of $L$. First, we prove that $N'$ is a loop. For this, let $(n_1,1),(n_2,1)\in N'$. Then $(n_1,1)\cdot(n_2,1)=(n_1n_2,1)\in N'$. Now consider the equation, 
		\begin{equation}\label{s2e1}
		(x,y)(n_1,1)=(n_2,1).
		\end{equation}
		Then $(x\phi_{y}(n_{1}), y) = (n_{2}, 1)$. This gives that $y=1$ and $x\phi_y(n_1)=n_2$. Therefore, $xn_1=n_2$. Since $N$ is a loop, we have a unique $x \in N$. It follows the unique solution of the Equation (\ref{s2e1}) in $N^{\prime}$. Using the similar argument, we get the unique solution for the equation $(n_1,1)(x,y)=(n_2,1)$  in $N^{\prime}$. Also, we have $(1,1)\in N'$ such that $(n,1)(1,1) = (n,1) = (1,1)(n,1)$ for all $(n,1)\in N^{\prime}$. Hence $N'$ is a loop.
		
		Now we show that $H'$ is a group. Let $(1,h_1)$, $(1,h_2)$, $(1,h_3) \in H'$. Then $(1,h_1)(1,h_2)=(1,h_1h_2)\in H^{\prime}$. Also, 
		\begin{align*}
		(1,h_1)\cdot(1,h_2)(1,h_3)&= (1,h_1)(1,h_2h_3)\\
		&= (1,h_1(h_2h_3))\\
		&= (1,(h_1h_2)h_3), \; (\text{as $H$ is a group})\\
		&= (1,h_1h_2)(1,h_3)\\
		&= (1,h_1)(1,h_2)\cdot(1,h_3).
		\end{align*}
		Thus, associativity holds in $H^{\prime}$. Clearly, $(1,1)\in H^{\prime}$ is the identity element and $(1,h)^{-1}=(1,h^{-1})\in H'$ for all $(1,h)\in H^{\prime}$. Thus $H'$ is a group.
		
		Now define the maps $f:H\rightarrow H'$ and $g:N\rightarrow N'$ by $h\mapsto (1,h)$ and $n\mapsto(n,1)$ respectively. One can easily observe that  $f$  is a group isomorphisms and $g$ is a loop isomorphism so, $H\simeq H'$ and $N\simeq N'$.
		\item[$(a)$] Using the part $(ii)$, let us identify $N$ and $H$  with their corresponding isomorphic images $N^{\prime}$ and $H^{\prime}$ in $L$. Then $(a)$ holds.
		\item[$(b)$] Let $(x,y)\in H\cap N$. Then clearly, $x = 1 = y$. Hence $(iv)$ holds.
		\item[$(c)$]  Let $(n,1)\in N$ and $(1,h)\in H$. Then as above
		\begin{align*}
		(1,h)((n,1)(1,h)^{-1}) &= (1,h)((n,1)(1,h^{-1}))\\
		&= (1,h)(n, h^{-1})\\
		&= (\phi_{h}(n), 1)
		\end{align*}
		Thus, using the identification, we get $\phi_{h}(n) = h(nh^{-1})$.
	\end{itemize}
\end{proof}

\begin{theorem}
	Let $L$  be a loop with a subloop $N$ and subgroup $H$ such that $N\cap H=\{1\}$ and the elements of $H$ associate with the elements of $N$. Let $\phi:H\rightarrow Sym(N)_1$ be a group homomorphism defined by $\phi_h(n)=h(nh^{-1})$. Then $NH\simeq N\rtimes_\phi H$.
\end{theorem}
\begin{proof}
	First, we show that $NH$ is a subloop of $L$. Let $n_{1}h_{1}, n_{2}h_{2}\in NH$. Then using the hypothesis that the elements of $H$ associates with the elements of $N$, we get
	\begin{align*}
	(n_{1}h_{1})(n_{2}h_{2}) &= (n_{1}h_{1})((n_{2}h_{1}^{-1}h_{1})h_{2})\\
	&= (n_{1}h_{1})(((n_{2}h_{1}^{-1})h_{1})h_{2})\\
	&= n_{1}(h_{1}((n_{2}h_{1}^{-1})h_{1}h_{2}))\\
	&= n_{1}((h_{1}(n_{2}h_{1}^{-1}))h_{1}h_{2})\\
	&= (n_{1}(h_{1}(n_{2}h_{1}^{-1})))(h_{1}h_{2}).
	\end{align*}
	Clearly $n_{1}(h_{1}(n_{2}h_{1}^{-1}))\in N$ and $h_{1}h_{2}\in H$. Thus $n_{1}h_{1}n_{2}h_{2}\in NH$ for all $n_{1}, n_{2}\in N$ and $h_{1}, h_{2}\in H$. Now, consider the equation 
	\begin{equation}\label{s2e2}
	X\cdot n_{1}h_{1} = n_{2}h_{2}.
	\end{equation} 
	Let $y = h_{2}h_{1}^{-1}\in H$ and $u\in N$ be the unique solution of the equation $x\cdot \phi_{h_{2}{h_{1}}^{-1}}(n_{1}) = n_{2}$. Then $uy\cdot n_{1}h_{1} = (u\phi_{y}(n_{1}))(yh_{1}) = (u\phi_{h_{2}h_{1}^{-1}}(n_{1}))((h_{2}h_{1}^{-1})h_{1}) = n_{2}h_{2}$. Thus $X = u(h_{2}{h_{1}}^{-1})\in NH$ gives the solution of the Equation (\ref{s2e2}). The uniqueness of $u$ gives the uniqueness of the solution $X$. Using the similar argument, we get the equation $n_{1}h_{1}\cdot X = n_{2}h_{2}$ has a unique solution in $NH$. Hence, $NH$ is a subloop of $L$. Now, we prove that each element of the subloop $NH$ can be written uniquely as the product of an element of $N$ and an element of $H$. Let $x \in NH$ be an element such that $x = n_{1}h_{1} = n_{2}h_{2}$. Then we have
	\begin{align*}
	(n_{1}h_{1})h_{1}^{-1} &= (n_{2}h_{2})h_{1}^{-1}\\
	n_{1}(h_{1}h_{1}^{-1})  &= n_{2}(h_{2}h_{1}^{-1})\\
	n_{1}&= n_{2}(h_{2}h_{1}^{-1})\\
	n_{2}^{-1}n_{1} &= n_{2}^{-1}(n_{2}(h_{2}h_{1}^{-1}))\\
	n_{2}^{-1}n_{1} &= (n_{2}^{-1}n_{2})(h_{2}h_{1}^{-1})\\
	n_{2}^{-1}n_{1} &= h_{2}h_{1}^{-1}.
	\end{align*}
	Since $N\cap H = \{1\}$, $n_{2}^{-1}n_{1} = h_{2}h_{1}^{-1} = 1$. This implies that $n_{1} = n_{2}$ and $h_{1} = h_{2}$. Therefore, each element of $NH$ can be uniquely expressed as $nh$, where $n\in N$ and $h\in H$. Let $n\in N$ and $h\in H$. Then there exists unique $n_{1}\in N$ and $h_{1}\in H$ such that $hn = n_{1}h_{1}$. Thus we have
	\begin{align*}
	n_{1}h_{1} &= (hn)h^{-1}h\\
	&= ((hn)h^{-1})h\\\
	&= (h(nh^{-1}))h.
	\end{align*}
	Using the uniqueness of representation, we get $n_{1} = h(nh^{-1})$ and $h_{1} =h$. Hence, $hn = (h(nh^{-1}))h$. Thus the map $f$ defined by $(n,h)\mapsto nh$ is a bijection of sets $NH$ and $N\times H$. Now 
	\begin{align*}
	f((n_1,h_1)(n_2,h_2)) &= f((n_1\phi_{h_1}(n_2),h_1h_2))\\ 
	&= (n_1\phi_{h_1}(n_2))(h_1h_2)\\ 
	&= (n_1(\phi_{h_1}(n_2)h_1))h_2\\ 
	&= (n_1((h_1(n_2h_1^{-1}))h_1))h_2\\ 
	&= (n_{1}(h_{1}n_{2}))h_{2}\\ 
	&= (n_1h_1)(n_2h_2)\\ 
	&= f((n_1,h_1))f((n_2,h_2)).
	\end{align*}
	This implies that $f$ is a loop homomorphism. Hence $NH\simeq N\rtimes H$.
\end{proof}

\begin{proposition}\label{s2p1}
	Let $L=N\rtimes_{\phi} H$ be a loop where $N$ and $H$ are groups. Then $L$ is a group if and only if  $\phi(H)$ is a subgroup of $Aut(N)$.
\end{proposition}
\begin{proof}
	Let $N$ and $H$ be groups such that $L=N\rtimes_{\phi} H$ is a loop. First, assume that $\phi(H)$ is a subgroup of $Aut(N)$. Then for all $(n_1,h_1),(n_2,h_2),(n_3,h_3)\in L$, we have
	\begin{align*}
	(n_1,h_1)((n_2,h_2)(n_3,h_3))&= (n_{1}, h_{1}) (n_{2}\phi_{h_{2}}(n_{3}), h_{2}h_{3})\\
	&= (n_{1}\phi_{h_{1}}(n_{2}\phi_{h_{2}}(n_{3})), h_{1}h_{2}h_{3})\\
	&= (n_{1}(\phi_{h_{1}}(n_{2}) \phi_{h_{1}}(\phi_{h_{2}}(n_{3}))), h_{1}h_{2}h_{3})\\
	&= ((n_{1}\phi_{h_{1}}(n_{2}))\phi_{h_{1}h_{2}}(n_{3}), h_{1}h_{2}h_{3})\\
	&= (n_{1}\phi_{h_{1}}(n_{2}), h_{1}h_{2})(n_{3}, h_{3})\\
	&=	(n_1,h_1)(n_2,h_2)\cdot(n_3,h_3).
	\end{align*}  	
	Thus $L$ is a group. Conversaly, let $L$ is a group. Then for all $n_{1}, n_{2}\in N$ and $h\in H$, we have
	\begin{align*}
	(\phi_{h}(n_{1}n_{2}), h) &= (1,h)(n_{1}n_{2}, h)\\
	&= (1,h)((n_{1}, 1)(n_{2}, 1))\\
	&= ((1,h)(n_{1}, 1))(n_{2}, 1)\\
	&= (\phi_{h}(n_{1}), h)(n_{2}, 1)\\
	&= (\phi_{h}(n_{1})\phi_{h}(n_{2}), h).
	\end{align*}
	Thus, $\phi_{h}(n_{1}n_{2}) = \phi_{h}(n_{1})\phi_{h}(n_{2})$. Hence $\phi_{h}\in Aut(N)$ and so, $\phi(H)$ is a subgroup of $Aut(N)$.
\end{proof}

%

\noindent	Let $L = N\rtimes_{\phi} H$ be a loop with corresponding group homomorphism $\phi: H \longrightarrow Sym(N)_{1}$. Then we define the set $Fix(\phi) = \{n\in N \mid \phi_{h}(n) = n \; \text{for all}\,\, h\in H\}$. Also, for each $n\in N$, we define a map $i_{n}: N\longrightarrow N$ by $i_{n}(n^{\prime}) = n(n^{\prime}n^{-1})$.

\begin{theorem}\label{s3t12}
	Let $L=N\rtimes_\phi H$ be a loop, where $N$ is a loop with left inverse property and $H$ is a group. Then the commutant of $L$ is $C(L)=\{(x,y)\mid x\in Fix(\phi), y\in Z(H) \, \text{and} \,\, \phi_y=i_{x^{-1}}\}$.
\end{theorem}	

\begin{proof}
	Let $S = \{(x,y)\mid x\in Fix(\phi), y\in Z(H) \, \text{and} \,\, \phi_y=i_{x^{-1}}\}$. Then for $(x,y)\in C(L)$, we have $(x,y)(a,b)=(a,b)(x,y)$ for all $(a,b)\in L$. Therefore,
	\begin{equation*}
	(x\phi_{y}(a),yb)= (a\phi_b(x),by).
	\end{equation*}
	Thus $x\phi_{y}(a) = a\phi_b(x)$ and $yb = by$. Clearly, $y\in Z(H)$. For $a=1$, we get $\phi_b(x) = x$, which implies that $x\in Fix(\phi)$. Since left inverse property holds in $N$, $x\phi_{y}(a) = ax$ implies that $\phi_{y}(a) = x^{-1}(ax) = i_{x^{-1}}(a)$. Thus $(x,y)\in S$.
	
	\noindent		Conversely, let $(x,y)\in S$. Then 
	\begin{align*}
	(x,y)(a,b) &= (x\phi_{y}(a), yb)\\
	&= (xi_{x^{-1}}(a), by)\\
	&= (x(x^{-1}(ax)), by)\\
	&= (ax, by)\\
	&= (a\phi_{b}(x), by)\\
	&= (a,b)(x,y).
	\end{align*}
	Therefore, $(x, y)\in C(L)$. Hence, $C(L) = S$.
\end{proof}	

\begin{theorem}\label{s3t13}
	Let $L=N\rtimes_\phi H$ be a loop. Then the middle nucleus of $L$ is  $N_\mu(L)=\{(x,y) \mid \phi_h(x\phi_y(n))=\phi_h(x)\phi_{hy}(n)\; \text{and}\; \phi_h(x)\in {N}_\mu(N)\,\, \forall h\in H\; \text{and}\; n\in N\}$.
\end{theorem}	
\begin{proof}
	Let $M = \{(x,y) \mid \phi_h(x\phi_y(n))=\phi_h(x)\phi_{hy}(n)\; \text{and}\; \phi_h(x)\in {N}_\mu(N)\,\, \forall h\in H\; \text{and}\; n\in N\}$. Then for $(x,y)\in M$ and $(n_1,h_1), (n_2,h_2)\in L$, we have
	\begin{align*}
	(n_1,h_1)((x,y)(n_2,h_2))=&(n_1,h_1)(x\phi_y(n_2),yh_2)\\
	=&(n_1\phi_{h_1}(x\phi_y(n_2)),h_1yh_2)\\
	=&(n_1\phi_{h_1}(x)\cdot \phi_{h_1y}(n_2),h_1yh_2)\\
	=& (n_{1}\phi_{h_{1}}(x), h_{1}y)(n_{2}, h_{2})\\
	=&((n_1,h_1)(x,y))(n_2,h_2).
	\end{align*}   
	Thus $(x,y)\in N_{\mu}(L)$.
	
	\noindent		Conversely, let $(x,y)\in N_{\mu}(L)$. Then for all $n\in N$ and $h\in H$, we have
	\begin{align*}
	(\phi_{h}(x\phi_{y}(n)), h) &= (1,h)(x\phi_{y}(n), 1)\\
	&= (1,h)((x,y)(n, y^{-1}))\\
	&= ((1,h)(x,y))(n, y^{-1})\\
	&= (\phi_{h}(x), hy)(n, y^{-1})\\
	&= (\phi_{h}(x)\phi_{hy}(n), h).
	\end{align*}
	Thus, we get $\phi_{h}(x\phi_{y}(n)) = \phi_{h}(x)\phi_{hy}(n)$. Now, for all $n_1, n_2\in N$ and $h\in H$, we have
	\begin{align*}
	(n_1(\phi_{h}(x)n_2), hy) &= (n_1(\phi_{h}(x)\phi_{hy}(n_2)), hy), \; (\text{where $\phi_{hy}(n) = n_{2}$ for some $n\in N$})\\
	&= (n_1\phi_{h}(x\phi_y(n)),hy)\\
	&= (n_{1}, h)(x\phi_{y}(n), y)\\
	&=(n_1,h)((x,y)(n,1))\\
	&=((n_1,h)(x,y))(n,1)\\
	&= (n_{1}\phi_{h}(x), hy)(n, 1)\\
	&=((n_1\phi_{h}(x))\phi_{hy}(n),hy)\\
	&= ((n_1\phi_{h}(x))n_{2}, hy).
	\end{align*}
	Therefore, $n_1(\phi_{h}(x)n_2)=(n_1\phi_{h}(x))n_2$. This implies that $\phi_{h}(x)\in N_\mu(N)$ for all $h\in H$. Thus $(x,y)\in M$. Hence, $N_{\mu}(L) = M$.
\end{proof}

\begin{theorem}\label{s3t14}
	Let $L=N\rtimes_\phi H$ be a loop. Then the right nucleus of $L$ is ${N}_\rho(L)=\{(x,y)\mid \phi_h(n\phi_{h'}(x))=\phi_h(n)\phi_{hh'}(x)\, \text{and}\;\phi_{h}(x)\in {N}_\rho(N) \, \forall h,h'\in H\, \text{and}\,\, n\in N\}$.
\end{theorem}
\begin{proof}
	Consider the set $R = \{(x,y)\mid \phi_h(n\phi_{h'}(x))=\phi_h(n)\phi_{hh'}(x)\, \text{and}\;\phi_{h}(x)\in {N}_\rho(N) \, \forall h,h'\in H\, \text{and}\,\, n\in N\}$. Then for $(x,y)\in R$ and $(n_{1},h_{1}),(n_{2},h_{2})\in L$, we have
	\begin{align*}
	(n_{1},h_{1})\cdot (n_{2},h_{2})(x,y)&= (n_{1}, h_{1}) (n_{2}\phi_{h_{2}}(x),h_{2}y)\\
	&=(n_{1}\phi_{h_{1}}(n_{2}\phi_{h_{2}}(x)),h_{1}h_{2}y)\\
	&=(n_{1}\cdot\phi_{h_{1}}(n_{2})\phi_{h_{1}h_{2}}(x),h_{1}h_{2}y)\\
	&= (n_{1}\phi_{h_{1}}(n_{2})\cdot\phi_{h_{1}h_{2}}(x),h_{1}h_{2}y), \; (\text{as $\phi_{h}(x)\in N_{\rho}(N) \;\forall h\in H$})\\
	&= (n_{1}\phi_{h_{1}}(n_{2}), h_{1}h_{2})\cdot (x,y)\\
	&=(n_{1},h_{1}) (n_{2},h_{2})\cdot(x,y)
	\end{align*}
	Thus $(x,y)\in N_\rho(L)$.
	
	\noindent		 Conversely, let $(x,y)\in N_\rho(L)$. Then for all $n\in N$ and $h_{1}, h_{2})\in H$, we have
	\begin{align*}
	(\phi_{h_{1}}(n\phi_{h_{2}}(x)),h_{1}h_{2}y) &= (1,h_{1})((n\phi_{h_{2}}(x),h_{2}y))\\
	&=(1,h_{1})((n,h_{2})(x,y))\\
	&=((1,h_{1})(n,h_{2}))(x,y)\\
	&= (\phi_{h_{1}}(n), h_{1}h_{2})(x,y)\\
	&=(\phi_{h_{1}}(n)\phi_{h_{1}h_{2}}(x),h_{1}h_{2}y).
	\end{align*}
	Therefore, $\phi_{h_{1}}(n\phi_{h_{2}}(x)) = \phi_{h_{1}}(n)(\phi_{h_{1}h_{2}}(x))$. Now, for $n_{1}, n_{2}\in N$ and $h\in H$, we have 
	\begin{align*}
	(n_{1}(n_{2}\phi_{h}(x)),hy) &= (n_{1},1)(n_{2}\phi_{h}(x), hy)\\
	&=(n_{1},1)((n_{2},h)(x,y))\\
	&=((n_{1},1)(n_{2},h))(x,y)\\
	&= (n_{1}n_{2}, h)(x,y)\\
	&=((n_{1}n_{2})\phi_{h}(x), hy).
	\end{align*}
	Therefore, $n_{1}(n_{2}\phi_{h}(x)) = (n_{1}n_{2})\phi_{h}(x)$. Thus $(x,y)\in R$. Hence, $N_{\rho}(L) = R$.
\end{proof}

\begin{theorem}\label{s3t15}
	Let $L=N\rtimes_\phi H$ be a loop. Then the left nucleus of $L$ is ${N}_\lambda(L)=\{(x,y)\mid x\phi_y(n\phi_{h}(n'))=x\phi_y(n)\cdot \phi_{yh}(n') \, \forall h\in H\, \text{and}\,\, n,n'\in N\}$.
\end{theorem}
\begin{proof}
	Consider the set $P = \{(x,y)\mid x\phi_y(n\phi_{h}(n'))=x\phi_y(n)\cdot \phi_{yh}(n') \, \forall h\in H\, \text{and}\,\, n,n'\in N\}$. Then for $(x,y)\in P$ and $(n_{1},h_{1}), (n_{2},h_{2})\in L$, we have
	\begin{align*}
	(x,y)((n_{1},h_{1})(n_{2},h_{2}))&= (x,y)(n_{1}\phi_{h_{1}}(n_{2}), h_{1}h_{2})\\
	&= (x\phi_y(n_{1}\phi_{h_{1}}(n_{2})),yh_{1}h_{2})\\
	&=((x\phi_y(n_{1}))\phi_{yh_{1}}(n_{2}),yh_{1}h_{2})\\
	&= (x\phi_y(n_{1}), yh_{1})(n_{2}, h_{2})\\
	&=((x,y)(n_{1},h_{2}))(n_{2},h_{2}).
	\end{align*}
	Thus $(x,y)\in N_\lambda(L)$.
	
	\noindent		 Conversely, let $(x,y)\in N_\lambda(L)$. Then for all $(n_{1},h_{1}), (n_{2},h_{2})\in L$, we have	
	\begin{align*}
	(x\phi_y(n_{1}\phi_{h_{1}}(n_{2})),yh_{1}h_{2})&=(x,y)((n_{1}\phi_{h_{1}}(n_{2}),h_{1}h_{2}))\\
	&= (x,y)((n_{1},h_{1})(n_{2},h_{2}))\\
	&= ((x,y)(n_{1},h_{1}))(n_{2},h_{2})\\
	&= (x\phi_{y}(n_{1}),yh_{1})(n_{2},h_{2})\\
	&= ((x\phi_y(n_{1}))\phi_{yh_{1}}(n_{2}),yh_{1}h_{2})
	\end{align*}	
	This implies that $x\phi_y(n_{1}\phi_{h_{1}}(n_{2}))=(x\phi_y(n_{1}))\phi_{yh_{1}}(n_{2})$. Thus $(x,y)\in P$. Hence, $P = N_{\lambda}(L)$.
\end{proof}

\begin{proposition}\label{s5p5}
	Let $L = N\rtimes_{\phi} H$ be a loop, where $N$ and $H$ are abelian groups. Then $C(L) = \{(x,y) \mid x\in Fix(\phi)\; \text{and}\; y\in \ker(\phi)\}$.
\end{proposition}
\begin{proof}
	Let $N$ and $H$ be abelian groups. Then using the Theorem \ref{s3t12}, for any $(x,y)\in C(L)$, we have $\phi_{y}(n) = i_{x}(n) = x^{-1}(nx) = n = I_{N}(n)$ for all $n\in N$, where $I_{N}$ is the identity map on $N$. Therefore, $y\in \ker(\phi)$. Since $H$ is abelian, $Z(H) = H$. Hence, $C(L) = \{(x,y) \mid x\in Fix(\phi)\; \text{and}\; y\in \ker(\phi)\}$.
\end{proof}
\begin{proposition}
	Let $L = N\rtimes_{\phi} H$ be a loop, where $N$ is an abelian group. Then ${N}_\rho(L)= N_\mu(L)$.
\end{proposition}
\begin{proof}
	Using the Theorems \ref{s3t13}, \ref{s3t14} and the fact that $N$ is a group, we get
	\begin{align*}
	N_\rho(L)&=\{(x,y)\mid \phi_h(n\phi_{h'}(x))=\phi_h(n)\phi_{hh'}(x) \forall h,h'\in H\, \text{and}\,\, n\in N\}\\
	\text{and}\quad N_\mu(L)&=\{(x,y) \mid \phi_h(x\phi_y(n))=\phi_h(x)\phi_{hy}(n)\, \forall h\in H\, \text{and}\,\, n\in N\}.
	\end{align*} 
	Now, let $(x,y)\in{N}_\rho(L)$. Then for all $h, h^{\prime}\in H$ and $n\in N$, we have 
	\begin{align*}
	\phi_{h}(x\phi_y(n))&= \phi_h(\phi_y(n)\phi_{1}(x))\\
	&=\phi_{hy}(n)\phi_{h1}(x)\\
	&=\phi_h(x)\phi_{hy}(n),\; (\text{as $N$ is an abelian group}). 
	\end{align*}
	Thus $(x,y)\in{N}_\mu(L)$ and so, ${N}_\rho(L)\subseteq{N}_\mu(L)$.
	
	\noindent	 Conversely, let $(x,y)\in{N}_\mu(L)$. Then for all $h, h^{\prime}\in H$ and $n\in N$, we have 
	\begin{align*}
	\phi_h(n)\phi_{hh^{\prime}}(x)&= \phi_{hh^{\prime}}(x)\phi_{h}(n),\; (\text{as $N$ is abelian})\\
	&= \phi_{hh^{\prime}}(x)\phi_{h}(\phi_{1}(n))\\
	&= \phi_{hh^{\prime}}(x)\phi_{h}(\phi_{(h^{\prime}y(h^{\prime}y)^{-1})}(n))\\
	&= \phi_{hh^{\prime}}(x)\phi_{h}(\phi_{h^{\prime}y}(\phi_{(h^{\prime}y)^{-1}}(n)))\\	
	&= \phi_{hh^{\prime}}(x)\phi_{h(h^{\prime}y)}(\phi_{(h^{\prime}y)^{-1}}(n))\\
	&= \phi_{hh^{\prime}}(x)\phi_{(hh^{\prime})y}(\phi_{(h^{\prime}y)^{-1}}(n))\\
	&= \phi_{hh^{\prime}}(x\phi_{y}(\phi_{(h^{\prime}y)^{-1}}(n))),\; (\text{as $(x,y)\in N_{\mu}(L)$})\\
	&= \phi_{h}(\phi_{h^{\prime}}(x\phi_{y}(\phi_{(h^{\prime}y)^{-1}}(n))))\\
	&= \phi_{h}(\phi_{h^{\prime}}(x)\phi_{h^{\prime}y}(\phi_{(h^{\prime}y)^{-1}}(n))),\; (\text{as $(x,y)\in N_{\mu}(L)$})\\
	&= \phi_{h}(\phi_{h^{\prime}}(x)\phi_{(h^{\prime}y)(h^{\prime}y)^{-1}}(n))\\
	&= \phi_{h}(\phi_{h^{\prime}}(x)\phi_{1}(n))\\
	&= \phi_{h}(\phi_{h^{\prime}}(x)n)\\
	&= \phi_{h}(n\phi_{h^{\prime}}(x)),\; (\text{as $N$ is abelian}).
	\end{align*}        
	Thus $(x,y)\in{N}_\rho(L)$ and so, ${N}_\mu(L)\subseteq{N}_\rho(L)$. Hence ${N}_\rho(L)={N}_\mu(L)$     
\end{proof}

\begin{proposition}\label{s5p1}
	Let $L = N\rtimes_{\phi} H$ be a loop, where $N$ is a group. Then ${N}_\rho(L)$ is a group.
\end{proposition}
\begin{proof}
	Let $L = N\rtimes_{\phi} H$ be a loop, where $N$ is a group. Then using the Theorem \ref{s3t14}, ${N}_\rho(L)=\{(x,y)\mid y\in H, \phi_h(n\phi_{h'}(x))=\phi_h(n)\phi_{hh'}(x), \; \forall h,h^{\prime}\in H \; \text{and}\; n\in N\}$. Let $(x_1,y_1)$, $(x_2,y_2)$, $(x_3,y_3)\in {N}_\rho(L)$. Then $(x_1,y_1)(x_2,y_2)=(x_1\phi_{y_1}(x_2),y_1y_2)$. Now, for all $h,h^{\prime}\in H$ and $n\in N$, we have 
	\begin{align*}
	\phi_{h}(n\phi_{h^{\prime}}(x_1\phi_{y_1}(x_2))) &= \phi_{h}(n\phi_{h^{\prime}}(x_{1})\phi_{h^{\prime}y}(x_{2}))\\
	&= \phi_{h}(n\phi_{h^{\prime}}(x_{1}))\phi_{hh^{\prime}y}(x_{2})\\
	&= \phi_{h}(n)\phi_{hh^{\prime}}(x_{1}) \phi_{hh^{\prime}y}(x_{2})\\
	&= \phi_{h}(n)\phi_{hh^{\prime}}(x_{1}\phi_{y}(x_{2})).
	\end{align*}
	Thus $(x_{1}, y_{1})(x_{2}, y_{2})\in N_{\rho}(L)$. Now, for associativity, we have
	\begin{align*}
	(x_1,y_1)(x_2,y_2)\cdot(x_3,y_3)=&(x_1\phi_{y_1}(x_2),y_1y_2)(x_3,y_3)\\
	=&(x_1\phi_{y_1}(x_2)\phi_{y_1y_2}(x_3),y_1y_2y_3)\\
	=&(x_1\phi_{y_1}(x_2\phi_{y_2}(x_3)),y_1y_2y_3)\\
	=&(x_1,y_1)(x_2\phi_{y_2}(x_3),y_2y_3)\\
	=&(x_1,y_1)\cdot(x_2,y_2)(x_3,y_3).
	\end{align*} 
	Moreover, we have $(1,1)\in {N}_\rho(L)$ such that $(x,y)(1,1) = (x,y) = (1,1)(x,y)$. Therefore, $(1,1)$ is the identity element of ${N}_\rho(L)$. Also, for each $(x,y)\in N_{\rho}(L)$, we have $(x,y)^{-1} = (\phi_{y^{-1}}(x^{-1}),y^{-1})\in N_{\rho}(L)$. Hence ${N}_\rho(L)$ is a group. 
\end{proof}

\begin{proposition}\label{s5p2}
	Let $L = N\rtimes_{\phi} H$ be a loop, where $N$ is a group. Then ${N}_\lambda(L)$ is a group.
\end{proposition}
\begin{proof}
	Let $L = N\rtimes_{\phi} H$ be a loop, where $N$ is a group. Then using the Theorem \ref{s3t15}, $N_\lambda(L)=\{(x,y)\mid \phi_y\in Aut(N)\}$. Now, let $(x_1,y_1)$, $(x_2,y_2)$, $(x_3,y_3)\in {N}_\lambda(L)$. Then $(x_1,y_1)(x_2,y_2)=(x_1\phi_{y_1}(x_2),y_1y_2)$. Since $\phi_{y_{1}}, \phi_{y_{2}} \in Aut(N)$, $\phi_{y_{1}}\circ \phi_{y_{2}} = \phi_{y_{1}y_{2}}\in Aut(N)$. Thus $(x_{1}, y_{1})(x_{2}, y_{2}) \in N_\lambda(L)$. Now for associativity, we have
	\begin{align*}
	(x_1,y_1)(x_2,y_2)\cdot(x_3,y_3)=&(x_1\phi_{y_1}(x_2),y_1y_2)(x_3,y_3)\\
	=&(x_1\phi_{y_1}(x_2)\phi_{y_1y_2}(x_3),y_1y_2y_3)\\
	=&(x_1\phi_{y_1}(x_2\phi_{y_2}(x_3)),y_1y_2y_3)\\
	=&(x_1,y_1)(x_2\phi_{y_2}(x_3),y_2y_3)\\
	=&(x_1,y_1)\cdot(x_2,y_2)(x_3,y_3).
	\end{align*} 
	Moreover, there exist $(1,1)\in N_{\lambda}(L)$ such that $(x,y)(1,1) = (x,y) = (1,1)(x,y)$ for all $(x,y)\in N_{\lambda}(L)$. Therefore, $(1,1)$ is the identity element of ${N}_\lambda(L)$. At last, for each $(x,y)\in N_{\lambda}(L)$, we have $(x,y)^{-1} = (\phi_{y^{-1}}(x^{-1}),y^{-1})\in N_{\lambda}(L)$. Hence ${N}_\lambda(L)$ is a group.
\end{proof}

\section{Isomorphic Semidirect Products}
In this section, we will give a criterion for the isomorphism between two loops obtained by the semidirect product of a loop and a group.
\begin{theorem}\label{s3t1}
	Let $L = N\rtimes_{\phi} H$ be a loop, where $N$ is a loop and $H$ is a group. Then for any decomposition $L=N_1\rtimes H_1$ with $N_1\simeq N$ and $H_1\simeq H$ there exists $f\in$ $Aut(L)$ such that $f(N_1)=N$ and $f(H_1)=H$.\\
	
	\noindent Let $L_1=N\rtimes_{\phi}H$ and  $L_2=N\rtimes_{\psi}H$ be two loops, where $N$ is a loop and $H$ is a group. Then $L_{1}$ is isomorphic to $L_{2}$ if and only if  there exist automorphisms $\alpha \in Aut(N)$ and $\beta \in Aut(H)$ such that $\alpha \circ \phi(h)\circ \alpha^{-1}=\psi(\beta(h))$ for all $h\in H$.
\end{theorem}
\begin{proof}
	First, assume that $L_1$ is isomorphic to $L_2$. Then there exists an isomorphism $\gamma$ from $L_{1}$ to $L_{2}$ such that $\gamma(L_1)=L_2$. Therefore $L_2=\gamma(N)\rtimes\gamma(H)$. Using the hypothesis, there exists $f\in Aut(L_2)$ such that $f(\gamma(N))=N$ and $f(\gamma(H))=H$. Since $f\circ\gamma$ is an isomorphism from $L_1$ to $L_2$, the map $f\circ \gamma$ induces certain automorphisms on $N$ and $H$. Let these automorphisms be $\alpha \in Aut(N)$ and $\beta \in Aut(H)$ such that $\alpha(N)=f(\gamma(N))=N$ and $\beta(H) = f(\gamma(H)) = H$.
	\vspace{.2cm}
	
	\noindent	Now, let $nh\in L_1$ be any element, where $n\in N$ and $h\in H$. Then  $f(\gamma(nh))= f(\gamma(n)\gamma(h))= f(\gamma(n))f(\gamma(h))=\alpha(n)\beta(h)$. Thus for any $n'h, nh'\in L_1$, we have		
	%
	\begin{align*}
	\alpha(n^{\prime})\alpha(\phi_{h}(n))\beta(hh^{\prime}) &=\alpha(n'\phi_h(n))\beta(hh')\\
	&= f(\gamma(n'\phi_h(n)hh'))\\
	&= f(\gamma((n'h)(nh')))\\
	&= f(\gamma(n'h)\gamma(nh'))\\ 
	&= f(\gamma(n'h))f(\gamma(nh'))\\ 
	&=(\alpha(n')\beta(h))(\alpha(n)\beta(h'))\\
	&=(\alpha(n')\psi_{\beta(h)}(\alpha(n)))(\beta(h)\beta(h'))\\
	&= (\alpha(n')\psi_{\beta(h)}(\alpha(n)))\beta(hh').
	\end{align*}
	
	\noindent Using the cancellation laws in $L_{2}$, we get $\alpha(\phi_h(n))=\psi_{\beta(h)}(\alpha(n))$ for all $n\in N$ and $h\in H$. Therefore, $\alpha \circ \phi_{h} = \psi_{\beta(h)}\circ \alpha$. This implies that $\alpha\circ \phi(h)\circ \alpha^{-1} = \psi(\beta(h))$ for all $h\in H$.
	\vspace{.2cm}
	
	\noindent 	Conversely, let us suppose that 
	$\alpha\in Aut(N)$ and $\beta\in Aut(H)$ such that $\alpha\circ \phi(h)\circ \alpha^{-1} = \psi(\beta(h))$ for all $h\in H$. Now, define a map $\eta : L_{1} \longrightarrow L_{2}$ by $(n,h) \mapsto (\alpha(n), \beta(h)$). Clearly, the map $\eta$ is well defined. Since the maps $\alpha$ and $\beta$ are bijections, $\eta$ is also a bijection. Now, for all $(n_{1},h_{1}), (n_{2},h_{2})\in L_1$, we have
	\begin{align*}
	\eta((n_{1},h_{1})\cdot(n_{2},h_{2}))&=\eta((n_{1}\phi_{h_{1}}(n_{2}),h_{1}h_{2}))\\
	&=(\alpha(n_{1}\phi_{h_{1}}(n_{2})),\beta(h_{1}h_{2}))\\
	&=(\alpha(n_{1})(\alpha\circ \phi_{h_{1}})(n_{2}),\beta(h_{1}h_{2}))\\
	&=(\alpha(n_{1})(\psi_{\beta(h_{1})}\circ\alpha)(n_{2}),\beta(h_{1}h_{2}))\\
	&=(\alpha(n_{1}) \psi_{\beta(h_{1})}(\alpha(n_{2})),\beta(h_{1})\beta(h_{2}))\\
	&=(\alpha(n_{1}),\beta(h_{1}))\cdot(\alpha(n_{2}),\beta(h_{2}))\\
	&=\eta((n_{1},h_{1}))\cdot \eta((n_{2},h_{2})).
	\end{align*}
	Thus $\eta$ is a loop homomorphism. Hence, $\eta$ is an isomorphism.
\end{proof}

\begin{theorem}\label{s4t2}
	Let $L_1=N\rtimes_{\phi}H$ and $L_2=N\rtimes_{\psi}H$ be two loops such that $\ker(\phi)=\{1\}= \ker(\psi)$. Then $L_1\simeq L_2$ if and only if $\phi(H)$ and $\psi(H)$ are conjugate in $Aut(N)$.
\end{theorem}
\begin{proof}
	First, let $L_1\simeq L_2$. Then by the Theorem \ref{s3t1}, there exists $\alpha\in Aut(N)$ and $\beta \in Aut(H)$ such that $\alpha\circ  \phi(h)\circ\alpha^{-1}=\psi(\beta(h))$ for all $h\in H$. Since $\beta(H)=H$, $\alpha\circ  \phi(H)\circ\alpha^{-1}=\psi(H)$. Thus $\phi(H)$ and $\psi(H)$ are conjugate in $Aut(N)$.
	\vspace{.2cm}
	
	\noindent	Conversely, let $\phi(H)$ and $\psi(H)$ are conjugate in $Aut(N)$. Then there exist $\alpha\in Aut(N)$ such that $\alpha \circ  \phi(H)\circ\alpha^{-1} = \psi(H)$. This defines a map $\sigma$ from $\phi(H)$ to $\psi(H)$ by $\phi(h) \mapsto \alpha \circ \phi(h) \circ \alpha^{-1}$ for all $h\in H$. Clearly, for all $h\in H$, $\sigma(\phi(h))$ is a bijection as both $\alpha$ and $\phi(h)$ are bijections. Also, $\sigma(\phi(h))(1) = (\alpha \circ \phi(h) \circ \alpha^{-1})(1) = 1$ for all $h\in H$. Therefore, $\sigma(\phi(h))\in Sym(N)_{1}$ and $\sigma(\phi(h))\in \psi(H)$.
	
	Now, let $h\in H$ such that $\sigma(\phi(h)) = I_{N}$, the identity map on $N$. Then $\phi(h)\circ \alpha = \alpha$ and so, $\phi(h) = 1$. Thus $h\in \ker(\phi) = \{1\}$. Therefore, $h = 1$ and so, $\ker(\sigma) = \{\phi(1) = I_{N}\}$. Thus, $\sigma$ is an injective map. Also, for any $\delta \in \psi(H)$ there exists $h\in H$ such that $\delta = \alpha\circ\phi(h)\circ\alpha^{-1}$. This proves that $\sigma$ is a bijection. In addition, for all $h,h^{\prime}\in H$, $\sigma(\phi(hh^{\prime})) = \alpha\circ \phi(hh^{\prime})\circ \alpha = \alpha\circ \phi(h) \phi(h^{\prime}) \circ \alpha^{-1} = (\alpha\circ \phi(h)\circ \alpha^{-1})(\alpha\circ \phi(h^{\prime})\circ \alpha^{-1}) = \sigma(\phi(h))\sigma(\phi(h^{\prime}))$. This shows that $\sigma$ is an isomorphism.
	
	Since $\ker(\phi) = \{1\} = \ker(\psi)$, the maps $\phi: H \longrightarrow \phi(H)$ and $\psi: H \longrightarrow \psi(H)$ are bijections. Thus the map $\psi^{-1}\circ \sigma \circ\phi: H \longrightarrow H$ is a bijection. Now $(\psi^{-1}\circ\sigma\circ \phi)(h_1h_2) = \psi^{-1}\circ\sigma (\phi(h_1h_2))=\psi^{-1}\circ\sigma(\phi(h_1)\phi(h_2)) = \psi^{-1}(\sigma(\phi(h_1)) \sigma(\phi(h_2))) = (\psi^{-1}\circ\sigma\circ \phi)(h_1)(\psi^{-1}\circ\sigma\circ \phi)(h_2)$. Therefore, $\psi^{-1}\circ\sigma\circ \phi \in Aut(H)$. Let $\beta = \psi^{-1}\circ\sigma\circ \phi$. Then for all $h\in H$, $\psi(\beta(h)) = \psi(\psi^{-1}\circ\sigma\circ \phi(h)) = \sigma(\phi(h)) = \alpha\circ \phi(h)\circ\alpha^{-1}$. Hence, $L_1\simeq L_2$. 
\end{proof}

%
%

\section{Split Meta-Cyclic Loop}
In this section, we will study the loops obtained by the semidirect product of two cyclic groups. An information about the left nucleus, middle nucleus, right nucleus and nucleus of such loops is given. Moreover, we will study the split metacyclic loops obtained by the semidirect product of two cyclic groups. In particular, we will study the loops $\mathbb{Z}_{p}\rtimes \mathbb{Z}_{m}$, where $p$ is an odd prime. The existence of such a loop is shown in the example below.
\begin{example}
	Let $\mathbb{Z}_5=\{0,1,2,3,4\}$ and $\mathbb{Z}_4=\{0,1,2,3\}$ be two cyclic groups of order five and four respectively. Then, we have loop structure $L_i=\mathbb{Z}_5\rtimes_{\phi_i}\mathbb{Z}_4$, where $\phi_i$ is a group homomorphism from $\mathbb{Z}_4$ to $Sym(\mathbb{Z}_5)_0$. Note that, the total number of such loops is equal to the total number of choices of the group homomorphisms $\phi_i: \mathbb{Z}_{4} \longrightarrow Sym(\mathbb{Z}_{5})_{0}\simeq S_{4}$, where $S_{4}$ is the symmetric group of degree 4. In this case, we get 16 choices of the homomorphisms $\phi_i$. Thus there are total 16 loops $L_{i} = \mathbb{Z}_5\rtimes_{\phi_i}\mathbb{Z}_4$ of order 20. Moreover, there are total 7 loops upto isomorphic. Among these there are 3 groups and 4 loops upto isomorphism. These are given as below,\\
	
	\noindent Case I - for $\phi_{1}(1)=I$, the identity map on the set $\mathbb{Z}_{5}$, the corresponding loop $L_1 = \mathbb{Z}_5\times_{\phi_1}\mathbb{Z}_4\simeq \mathbb{Z}_{20}$ is a cyclic group of order 20.\\
	
	\noindent Case II- for  $\phi_2(1)=(1243)$ and $\phi_3(1) = (1342)$, the corresponding loops $L_2$ and $L_3$ are isomorphic as groups.\\
	
	\noindent Case III- for $\phi_4(1)= (14)(23)$, the corresponding loop $L_4$ is a group.\\
	
	\noindent Case IV- for $\phi_5(1)=(1234)$, $\phi_{6}(1) = (1324)$, $\phi_{7}(1) = (1432)$ and $\phi_{8}(1) = (1423)$, the corresponding loops $L_5, L_6, L_7$ and $L_8$ are isomorphic as loop.\\
	
	\noindent Case V- for $\phi_9(1)=(12)$, $\phi_{10}(1) = (13)$, $\phi_{11}(1) = (24)$ and $\phi_{12}(1) = (34)$, the corresponding loops $L_9, L_{10}, L_{11}$ and $L_{12}$ are isomorphic as loops.\\
	
	\noindent Case VI- for  $\phi_{13}(1)=(14)$ and $\phi_{14}(1) = (23)$, the corresponding loops $L_{13}$ and $L_{14}$ are isomorphic as loops.\\
	
	\noindent Case VII- for  $\phi_{15}(1)=(12)(34)$ and $\phi_{16}(1) = (13)(42)$, the corresponding loops $L_{15}$ and $L_{16}$ are isomorphic as loops.\\
	
	\noindent Using GAP \cite{gap}, we get
	
	\begin{table}[h!]
		\begin{tabular}{|c|c|c|c|c|c|c|}
			\hline
			{Case} & ${N}_\lambda(L)$  & ${N}_\rho(L)$ & ${N}_\mu(L)$ & ${N}(L)$ &  ${C}(L)$ & ${Z}(L)$ \\
			\hline
			IV & $\mathbb{Z}_{5}$ & $\mathbb{Z}_{4}$ & $\mathbb{Z}_{4}$ & Trivial & Trivial & Trivial \\
			\hline
			V & $\mathbb{Z}_{10}$ & $\mathbb{Z}_{4}$ & $\mathbb{Z}_{4}$ & $\mathbb{Z}_{2}$ & Size $= 6$ & $\mathbb{Z}_{2}$ \\
			\hline
			VI & $\mathbb{Z}_{10}$ & $\mathbb{Z}_{10}$ & $\mathbb{Z}_{4}$ & $\mathbb{Z}_{2}$ & Size $= 6$ & $\mathbb{Z}_{2}$ \\
			\hline
			VII & $\mathbb{Z}_{10}$ & $\mathbb{Z}_{4}$ & $\mathbb{Z}_{4}$ & $\mathbb{Z}_{2}$ & $\mathbb{Z}_{2}$ & $\mathbb{Z}_{2}$\\
			\hline
		\end{tabular}
		\caption{Nuclei, commutant and center of the loops in Case IV - VII.}
	\end{table}
\end{example} 

\noindent The above example motivates us to study loops which are the semidirect product of two cyclic groups defined as split metacyclic loops. Let $\mathbb{Z}_{n}$ and $\mathbb{Z}_{m}$ be two cyclic groups of order $n$ and $m$ respectively and $\phi : \mathbb{Z}_n \longrightarrow Sym(\mathbb{Z}_{m})_0$ be a group homomorphism. Then $L=\mathbb{Z}_{m}\rtimes_\phi\mathbb{Z}_n$ is a metacyclic loop. 
Now, we study the metacyclic loops which are not groups as follows.

\begin{proposition}\label{s5p3}
	Let $L=\mathbb{Z}_{m}\rtimes_{\phi}\mathbb{Z}_p$ be a meta-cyclic loop which is not a group. Then ${N}_\lambda(G)\simeq \mathbb{Z}_{m}$.
\end{proposition}
\begin{proof}
	Let $(x, y)\in N_{\lambda}(L)$ be any element. Then using the Theorem \ref{s3t15}, we get 
	\begin{equation}\label{e1}
	\phi_{y}(n\phi_{h}(n^{\prime})) = \phi_{y}(n)\phi_{yh}(n^{\prime}).
	\end{equation}
	Now, for all $n_{1}, n_{2}\in N$, we have
	\begin{align*}
	\phi_{y}(n_{1}n_{2}) &= \phi_{y}(n_{1}\phi_{h}(n^{\prime})),\; (\text{where $n_{2} = \phi_{h}(n^{\prime})$, for some $n^{\prime}\in N$ and $h\in H$})\\
	&= \phi_{y}(n_{1})\phi_{yh}(n^{\prime}),\; (\text{using the Equation (\ref{e1})})\\
	&= \phi_{y}(n_{1})\phi_{y}(n_{2}).
	\end{align*}
	Thus $\phi_{y}\in Aut(N)$ and so, $N_\lambda(L)=\{(x,y)\mid \phi_y\in Aut(N)\}$. Note that if $(x,y)\in N_{\lambda}(L)$ for any non-identity element $y\in \mathbb{Z}_{p}$, then $L = N_{\lambda}(L)$. But this is not possible as $L$ is a not a group. Therefore, $N_{\lambda}(L) = \{(x,y) \mid y = 0\; \text{and}\; x\in \mathbb{Z}_{m}\}\simeq \mathbb{Z}_{m}$. Hence ${N}_\lambda(G)\simeq \mathbb{Z}_{m}$.
\end{proof}
\begin{proposition}\label{s5p4}
	Let $L=\mathbb{Z}_{m}\rtimes_{\phi}\mathbb{Z}_{p}$ be a split meta-cyclic loop which is not a group. Then $|N_{\rho}(L)| = pr$, where $r\ne m$ and $r$ is a divisor of $m$. 
\end{proposition}
\begin{proof}
	Using the Theorem \ref{s3t14}, we have ${N}_\rho(L)=\{(x,y)\mid  \phi_h(n\phi_{h'}(x))=\phi_h(n)\phi_{hh'}(x), \; \forall h, h^{\prime}\in \mathbb{Z}_{p}\; \text{and}\; n\in \mathbb{Z}_{m}\}$. Clearly, $(0,y)\in N_{\rho}(L)$ for all $y\in \mathbb{Z}_{p}$. Also, using the Proposition \ref{s5p1}, $|N_{\rho}(L)| = pr$, where $r$ divides $m$. Since $L$ is not a group, $N_{\rho}(L) \ne L$. Therefore, $r\ne m$. Hence, the proof.
\end{proof}
\begin{corollary}
	Let $L=\mathbb{Z}_{m}\rtimes_{\phi}\mathbb{Z}_{p}$ be a split meta-cyclic loop which is not a group. Then $N(L)$ is a group and $|N(L)| = r$, where $r\ne m$ and $r$ is a divisor of $m$. 
\end{corollary}
\begin{proof}
	Using the Propositions \ref{s5p3} and \ref{s5p4}, we have 
	\[N(L) = \{(x,0) \mid \phi_h(n\phi_{h'}(x))=\phi_h(n)\phi_{hh'}(x),\; \forall h, h^{\prime}\in \mathbb{Z}_{p}\; \text{and}\; n\in \mathbb{Z}_{m}\}.\]
	Since $N_{\lambda}(L)$  and $N_{\rho}(L)$ are groups, $N(L)$ is a group. It is evident from the description of $N(L)$ that $N(L)$ is isomorphic to a proper subgroup of $\mathbb{Z}_{m}$. Hence $|N(L)| = r$, where $r\ne m$ and $r$ is a divisor of $m$.
\end{proof}
As an illustration, we will study the metacyclic loops of order $3^{3}$ below.

\begin{example}
	Let $L=\mathbb{Z}_9\rtimes_{\phi}\mathbb{Z}_3$ be a meta cyclic loop of order $3^{3}$, where $\phi$ is a group homomorphism from $\mathbb{Z}_3$ to $Sym(\mathbb{Z}_9)_0$. Then the total number of these split meta-cyclic loops is equal to $|Hom(\mathbb{Z}_3, Sym(\mathbb{Z}_9)_0)|=1233$. Using GAP \cite{gap}, we see that there are only $111$ meta-cyclic loops (with the given construction) upto isomorphism. Among these loops, there are only 2 groups upto isomorphism and 109 loops upto isomorphism. Now, we give the description of the nuclei, commutant and center of these loops below.
	\begin{itemize}
		\item[$(i)$]  For each loop $L$, $N_{\lambda}(L) \simeq \mathbb{Z}_9$.
		\item[$(ii)$]  Only $2$ loops have right nucleus isomorphic to $\mathbb{Z}_{3}\times \mathbb{Z}_{3}$ and all other loops have right nucleus isomorphic to $\mathbb{Z}_{3}$.
		\item[$(iii)$] For each loop $L$, $N(L) = Z(L)$. Only 2 loops have nucleus isomorphic to $\mathbb{Z}_{3}$ and rest all loops have trivial nucleus.
		\item[$(iv)$]  For each loop $L$, the associator is isomorphic to $\mathbb{Z}_9$.
		\item[$(v)$] The commutant of $10$ loops have order $6$ and all other loops have commutant of order $3$.
	\end{itemize} 
\end{example}

	\end{document}